\documentclass{amsart}
\usepackage{amssymb, amsmath, latexsym, amscd, graphicx, color}
\usepackage[all]{xy}
\usepackage[dvipdfm]{hyperref}
\usepackage[all]{hypcap}
\newtheorem{theorem}{Theorem}[section]

\newtheorem{definition}[theorem]{Definition}

\newtheorem{proposition}[theorem]{Proposition}
\def\pagenumber{1}
\begin{document}
\setcounter{page}{\pagenumber}
\newcommand{\T}{\mathbb{T}}
\newcommand{\R}{\mathbb{R}}
\newcommand{\Q}{\mathbb{Q}}
\newcommand{\N}{\mathbb{N}}
\newcommand{\Z}{\mathbb{Z}}
\newcommand{\tx}[1]{\quad\mbox{#1}\quad}
\parindent=0pt
\def\SRA{\hskip 2pt\hbox{$\joinrel\mathrel\circ\joinrel\to$}}
\def\tbox{\hskip 1pt\frame{\vbox{\vbox{\hbox{\boldmath$\scriptstyle\times$}}}}\hskip 2pt}
\def\circvert{\vbox{\hbox to 8.9pt{$\mid$\hskip -3.6pt $\circ$}}}
\def\IM{\hbox{\rm im}\hskip 2pt}
\def\ES{\vbox{\hbox to 8.9pt{$\big/$\hskip -7.6pt $\bigcirc$\hfil}}}
\def\TR{\hbox{\rm tr}\hskip 2pt}
\def\TRS{\hbox{\rsmall tr}\hskip 2pt}
\def\GRAD{\hbox{\rm grad}\hskip 2pt}
\def\GRADS{\hbox{\rsmall grad}\hskip 2pt}
\def\bull{\vrule height .9ex width .8ex depth -.1ex}
\def\VLLA{\hbox to 25pt{\leftarrowfill}}
\def\VLRA{\hbox to 25pt{\rightarrowfill}}
\def\DU{\mathop{\bigcup}\limits^{.}}
\setbox2=\hbox to 25pt{\rightarrowfill}
\def\DRA{\vcenter{\copy2\nointerlineskip\copy2}}
\def\RANK{\hbox{\rm rank}\hskip 2pt}
\font\rsmall=cmr7 at 7truept \font\bsmall=cmbx7 at 7truept
\newfam\Slfam
\font\tenSl=cmti10 \font\neinSl=cmti9 \font\eightSl=cmti8
\font\sevenSl=cmti7 \textfont\Slfam=\tenSl
\scriptfont\Slfam=\eightSl \scriptscriptfont\Slfam=\sevenSl
\def\Sl{\fam\Slfam\tenSl}

\vskip -1cm
\title[Exotic Heat PDE's]{\mbox{}\\[1cm] EXOTIC HEAT PDE's}
\author{Agostino Pr\'astaro}
\maketitle
\vspace{-.5cm}
{\footnotesize
\begin{center}
Department of Methods and Mathematical Models for Applied
Sciences, University of Rome ''La Sapienza'', Via A.Scarpa 16,
00161 Rome, Italy. \\
E-mail: {\tt Prastaro@dmmm.uniroma1.it}
\end{center}
}
\vspace{1cm}
\centerline{\it This paper is dedicated to Stephen Smale in occasion of his 80th birthday}

{\footnotesize
\vspace{1cm} {\bf ABSTRACT.} Exotic heat equations that allow to prove the Poincar\'e conjecture, some related problems and suitable generalizations too are considered. The methodology used is the PDE's algebraic topology, introduced by A. Pr\'astaro in the geometry of PDE's, in order to characterize global solutions.\footnote{This paper has been published in \cite{PRA16}.\\Work partially supported by Italian grants MURST ''PDE's Geometry and Applications''.}

\vskip 0.5cm

{\bf AMS (MOS) MS CLASSIFICATION. 55N22, 58J32, 57R20; 20H15.}
{\rm KEY WORDS AND PHRASES. Integral bordisms in PDE's;
Existence of local and global solutions in PDE's; Conservation laws;
Crystallographic groups; Singular PDE's; singular MHD-PDE's.}
}

\section{\bf Introduction}

{\footnotesize\hfill\rightline{\it ''In a category\hskip 5pt $\mathcal{C}$ of manifolds}

\hfill\rightline{\it every homotopy sphere is homeomorphic to a sphere.''}

\begin{figure}[h]
\scalebox{0.9}{$\centerline{\includegraphics[height=2.5cm]{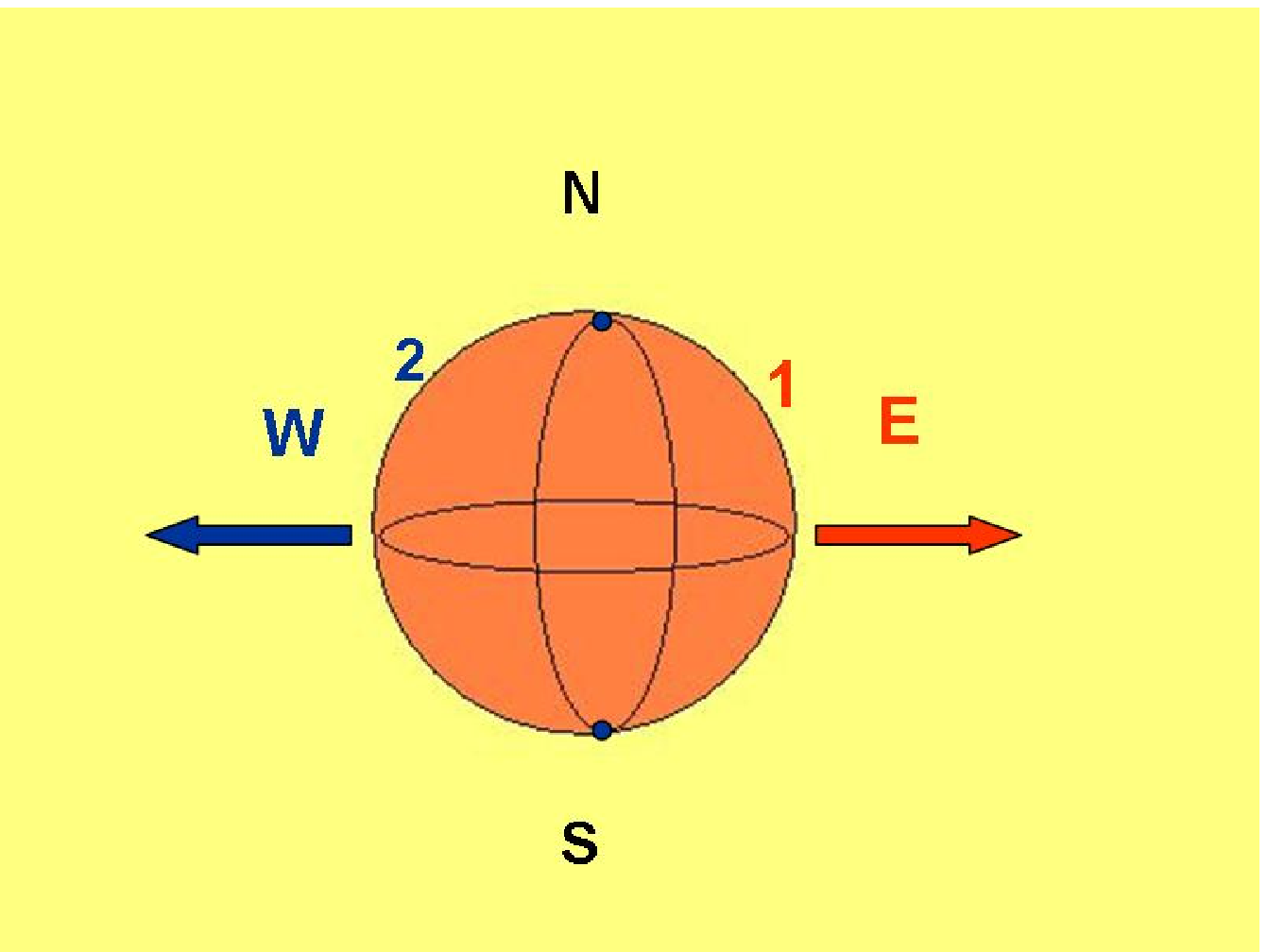}
\includegraphics[height=2.5cm]{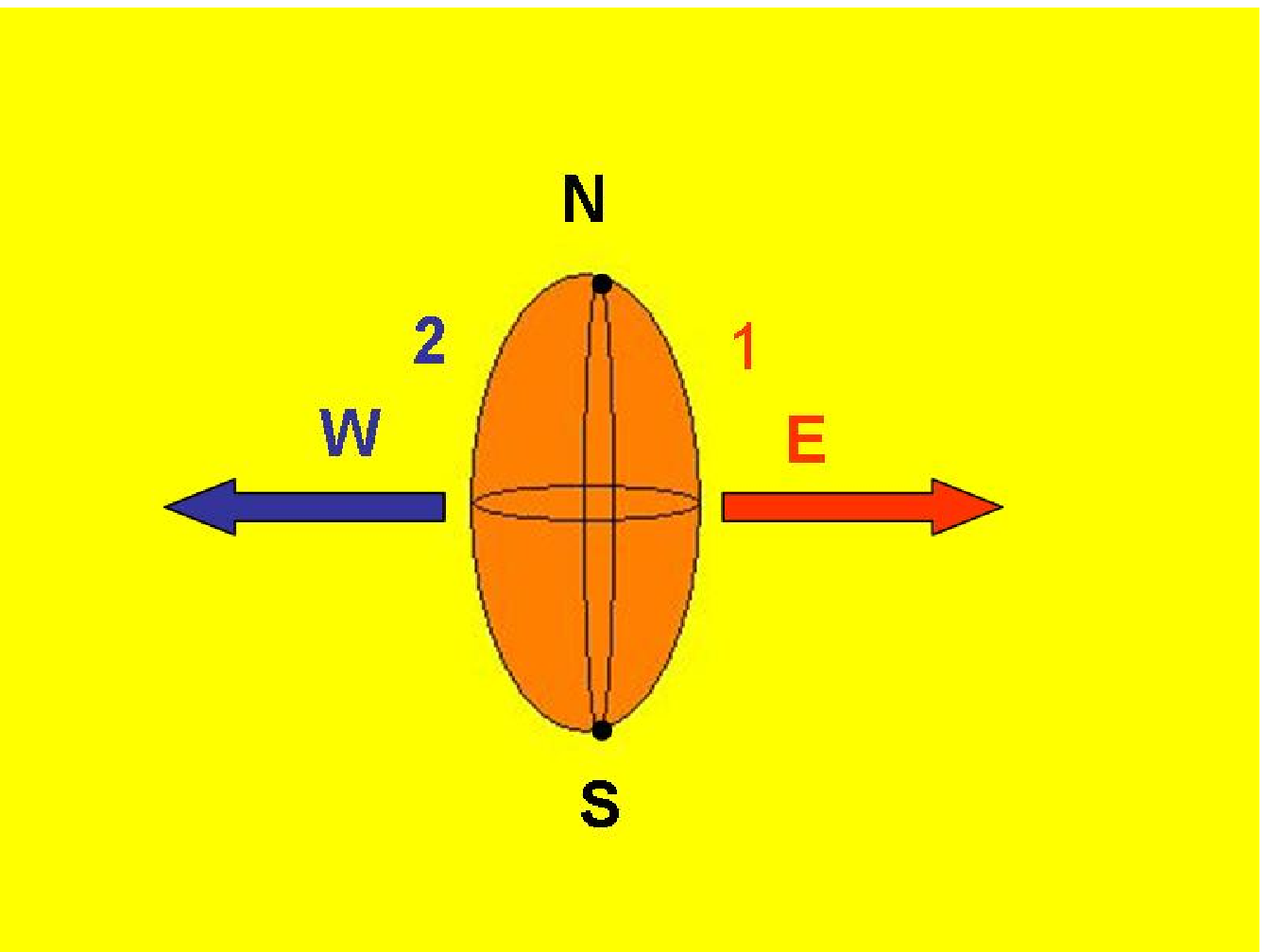} \includegraphics[height=2.5cm]{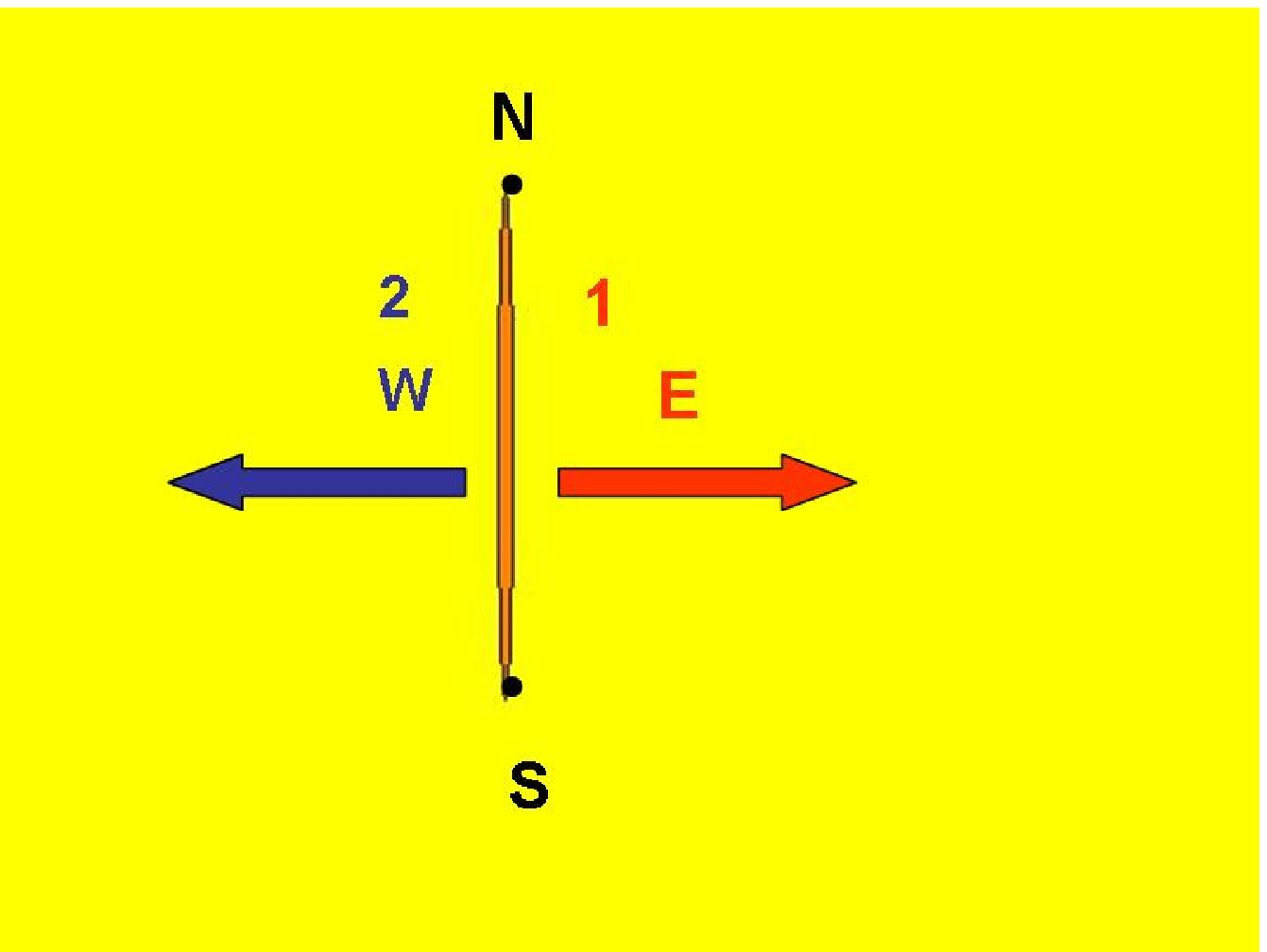}
\includegraphics[height=2.5cm]{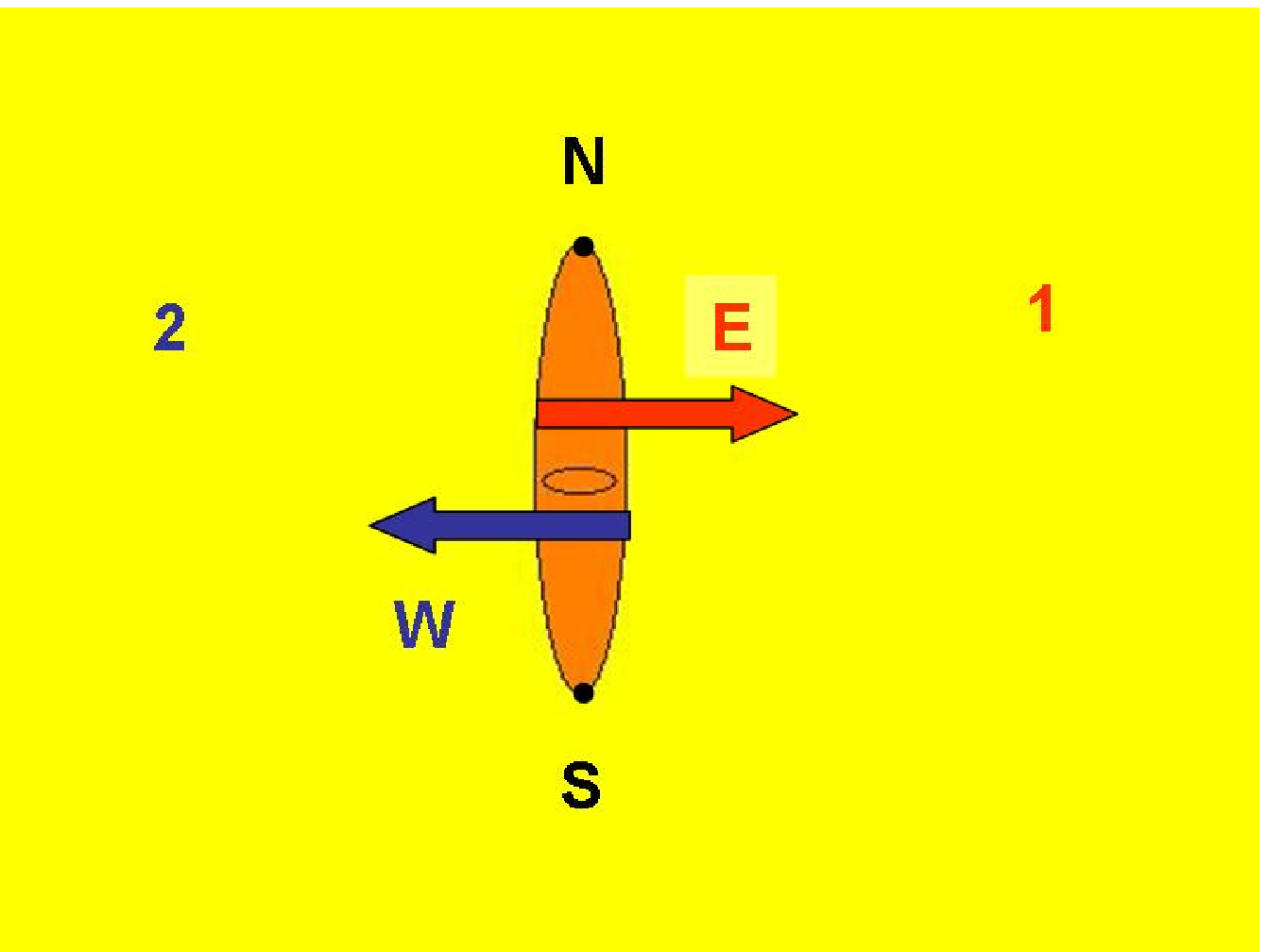}}$}
\caption{Smale's paradox: turning a sphere $S^2\subset\mathbb{R}^3$ inside out.}
\label{smale-paradox}
\end{figure}
}
\vskip 1cm

The famous Poincar\'e's conjecture is about $n$-dimensional manifolds, with $n=3$, (R. S. Hamilton \cite{HAMIL1, HAMIL2, HAMIL3, HAMIL4, HAMIL5}, G. Perelman \cite{PER1, PER2}, A. Pr\'astaro \cite{PRA14, AG-PRA1}), but there are also generalizations of this conjecture for higher dimension manifolds. For $n=4$ the Poincar\'e conjecture has been proved by M. Freedman \cite{FREEDMAN, FREEDMAN-QUINN}) and for $n\ge 5$, by S. Smale \cite{SMALE1, SMALE2}. More recently has been given a generalization for quantum supermanifolds by A. Pr\'astaro, that has also proved it in \cite{PRA12}). Nowadays, one can state that a generalized Poincar\'e conjecture can be proved, or disproved, depending on the particular category $\mathcal{C}$ in which it is formulated. This problem aroses in the framework of the geometric topology, but in order to be solved it was necessary to go outside that framework and recast the problem in a theory of PDE's. But the more recent results by A. Pr\'astaro \cite{PRA12, PRA14, AG-PRA1}), have proved that it was necessary to return inside the algebraic topology framework, applied to the PDE's geometric theory. Really, it was soon evident that remaining in a pure algebraic topologic approach it was not enough to solve this conjecture. In fact, a fundamental idea to solve this problem is to ask whether it is possible find a smooth manifold, $V$, that without singular points bords a $3$-dimensional compact, closed, smooth, simply connected manifold $N$ with $S^3$, when $N$ is  homotopy equivalent to $S^3$. The bordism theory is able to state that a smooth manifold $V$ such that $\partial V=N\bigcup S^3$, there exists, since the nonoriented and oriented $3$-dimensional bordism groups $\Omega_3$ and ${}^+\Omega_3$ respectively, are both trivial: $\Omega_3={}^+\Omega_3=0$. However, by simply looking to the above bordism groups it is impossible to state if $V$ has singular points (i.e., has holes) or it is a cylinder. By the way, more informations can be obtained by the {\em h-cobordism theory}. More precisely the {\em h-cobordism theorem} in a category $\mathcal{C}$ of manifolds, states that if the compact manifold $V$ has $\partial V=N_0\sqcup N_1$, such that the inclusion maps $N_i\hookrightarrow V$, $i=0,1$, are homotopy equivalences, (i.e., $V$ is a {\em h-cobordism}), and $\pi_1(N_i)=0$, then $V\cong_{\mathcal{C}} N_0\times[0,1]$. This theorem holds for $n\ge 5$ in the category of smooth manifolds (S. Smale) and for $n=4$ in the category of topological manifolds (M. Freedman). But it does not work for $n=3$!

A very important angular stone, in the long history about the solution of the Poincar\'e conjecture, has been the introduction, by R. S. Hamilton, of a new approach recasting the problem in to solving a PDE, the {\em Ricci flow equation}, and asking for nonsingular solutions there, that starting from a Riemannian manifold $(N,\gamma)$ arrive to the $3$-dimensional sphere $S^3$, respectively identified with initial and final Cauchy manifolds in the Ricci flow equation. In that occasion the Mathematical Analysis, or more precisely the Functional Analysis, entered in the Poincar\'e conjecture problem. This approach has had many improvements until the papers by G. Perelman.
More recently, A. Pr\'astaro, by using his algebraic topologic theory of PDE's was able to give a pure geometric proof of the Poincar\'e conjecture. Let us emphasize that the usual geometric methods for PDE's (Spencer, Cartan), were able to formulate for nonlinear PDE's, local existence theorems only, until the introduction, by A. Pr\'astaro, of the algebraic topologic methods in the PDE's geometric theory. These give suitable tools to calculate integral bordism groups in PDE's, and to characterize global solutions. Then, on the ground of integral bordism groups, a new geometric theory of stability for PDE's and solutions of PDE's has been built. These general methodologies allowed to A. Pr\'astaro to solve fundamental mathematical problems too, other than the Poincar\'e conjecture and some of its generalizations, like characterization of global smooth solutions for the Navier-Stokes equation and global smooth solutions with mass-gap for the quantum Yang-Mills superequation. (See \cite{PRA3, PRA4, PRA5, PRA6, PRA7, PRA8, PRA9, PRA10, PRA11, PRA12, PRA13, PRA15}.\footnote{See also Refs. \cite{AG-PRA1, AG-PRA2, PRA-RAS}, where interesting related applications of the PDE's Algebraic Topology are given.})

The main purpose of this paper is to emphasize some problems related to {\em exotic heat PDE's} recently focused.\footnote{The Ricci flow equation can be considered a generalization of the classical Fourier's heat equation $u_t-\kappa u_{xx}=0$. In this paper we call {\em exotic heat equations}, PDE's that, like the Ricci flow equation, are of the type $F^j\equiv u^j_t-f^j(u^i_k)=0$, $1\le i,j\le m$, with the length $|k|$ of the multi-index $k\in\{1,\cdots,n\}$, given by  $0\le |k|\le s\le r$, where $F^j:J^r(W)\to \mathbb{R}$ are analytic functions of order $r\ge 0$ on a fiber bundle $\pi:W\equiv \mathbb{R}\times E(M)\to \mathbb{R}\times M$, where $E(M)$ is a vector bundle over $M$, with $M$ an analytic manifold of dimension $n$. Let us emphasize that the structure of exotic heat equation is the more suitable to use in order to prove (generalized) Poincar\'e conjectures. In fact, the idea to use PDE's to solve the Poincar\'e's conjecture, was the initial motivation to introduce and study the well-known {\em Yamabe equation} \cite{YAMABE}. But that road did not turn out a lucky choice to prove the conjecture, even if the Yambe equation is a very important equation to study conformal problems in Riemannian geometry. (See, e.g., \cite{AG-PRA1}.)}

Let us conclude this Introduction with a theorem, direct issue from results contained in \cite{PRA14}.
\begin{theorem}\label{homeomorphic-three-manifold-diffeomorphic-three-sphere}
Any $3$-dimensional compact, closed simply connected smooth manifold, $M$, homotopy equivalent to $S^3$, is diffeomorphic to $S^3$.\footnote{This last result agrees with the Hauptvermuntung conjecture that was proved for $(n=2,3)$-dimensional manifolds and disproved for $(n\ge 4)$-dimensional manifolds. (See \cite{CASSON, MILNOR, MOISE1, MOISE2, RADO, SULLIVAN, TUSCH}.)}
\end{theorem}

\begin{proof}
In fact, the method followed by A. Pr\'astaro to prove the Poincar\'e conjecture \cite{PRA14, AG-PRA1}, allows us to conclude that in dimension $n=3$, under the hypotheses of the theorem, one can state that the manifold $V$, smooth solution of the Ricci flow equation, $(\partial t.g_{ij})-\kappa R_{ij}(g)=0$, such that $\partial V=M\sqcup S^3$, has the following structure $V\cong M\times[0,1]$, hence must necessarily be $M\cong S^3$, (in the category of smooth manifolds). In other words the algebraic topology of PDE's, introduced by A. Pr\'astaro, allows us to complete h-cobordism theorem in the category of $3$ dimensional smooth manifolds.
\end{proof}

\section{\bf RICCI FLOW EQUATION IN NON-COMPACT CASE}

In this section we consider the problem of existence and uniqueness of smooth complete solutions of the Ricci flow equation, in the noncompact case. In particular we answer to a recent question put for the $n$-dimensional Euclidean case in \cite{CHEN}. Let us consider the Ricci flow equation, written in the form

\begin{equation}\label{ricci-flow-equation}
\fbox{$(\partial t.g_{ij})(x,t)=-2\, R_{ij}(x,t),\quad (x,t)\in M\times\mathbb{R}$}
\end{equation}

when $M$ is a $n$-dimensional smooth, noncompact and complete Riemannian manifold, $(M,\gamma)$.  In particular the assumption that $M$ is not compact and complete, is a substantial difference with respect to the usual way in which such equation has been considered. In fact, the Ricci flow equation has been principally introduced by R. S. Hamilton, to prove the Poincar\'e conjecture on compact closed $3$-dimensional manifolds homotopic equivalent to $S^3$. (See \cite{{CHOW-CHU-GLIE-GUE-ISE-IVEY-KNOP-LU-LUO-NI-1}, {CHOW-CHU-GLIE-GUE-ISE-IVEY-KNOP-LU-LUO-NI-2}, HAMIL1, HAMIL2, HAMIL3, HAMIL4, HAMIL5, PER1, PER2} to follows the proof of the Poincar\'e conjecture in the approach  by Hamilton-Perelmann. Furthermore, in \cite{PRA14, AG-PRA1} the Poincar\'e conjecture has been proved by using a method different by the Hamilton-Perelmann's one.)

The main motivations to write this section are some recent works where the Ricci flow equation is implemented on a noncompact complete Riemannian manifold, and solutions are request to identify complete metrics on the sectional noncompact manifolds. (See, e.g., \cite{CHEN, SHI}.) More specifically, it was very intriguing the recent paper by B.-L. Chen \cite{CHEN}, where it is emphasized the uniqueness of such solutions in the dimensional $3$, and the following question: ''{\em Does the strong uniqueness of the Ricci flow hold on the Euclidean space $\mathbb{R}^n$, for $n\ge 4$\hskip 2pt ?}\hskip 3pt''.

In fact we have the following theorem.
\begin{theorem}\label{complete-noncompact-case-ricci-flow-equation}
Let us consider the Ricci flow equation {\em(\ref{ricci-flow-equation})}
on the Euclidean space $(\mathbb{R}^n,\gamma)$. Then, the unique solution of {\em(\ref{ricci-flow-equation})} for complete metrics is just $\gamma$.
\end{theorem}

\begin{proof}
Let us also underline that the restriction to consider noncompact complete metrics does not permit flows with contractions, and neither simple homeomorphic ones. Isometric diffeomorphisms are instead permitted. Therefore, let us add to the equation (\ref{ricci-flow-equation}) the condition that $g(x,t)$ is realized by means of a one parameter set of diffeomorphisms $\phi_t:M\to M$, by means of pull-back, i.e, put $g(x,t)=\phi_t^*\gamma$. One has for the Ricci tensor $R_{ij}(x)$ of $\gamma(x)$, the following induced deformation $R(x,t)=R(g(x,t))=\phi_t^*R(\gamma(x))$. In fact, for the natural covariance of the Riemannian metric and its Ricci tensor, one has the following commutative diagram

\begin{equation}
\scalebox{0.7}{$\xymatrix@C=80pt@R=80pt{M\ar@/^3pc/[rrr]_{R(g(x,t))}\ar@/^2pc/[rr]_{D^2g(x,t)}\ar@/^1pc/[r]_{g(x,t)}\ar[d]_{\phi_\lambda}&S^0_2M\ar[l]&J{\it D}^2(S^0_2M)\ar[l]\ar[r]_{R}&S^0_2M\\
M\ar@/_3pc/[rrr]^{R(\gamma(x))}\ar@/_2pc/[rr]^{D^2\gamma(x)}\ar@/_1pc/[r]^{\gamma}&S^0_2M\ar[u]^{S^0_2(\phi_\lambda)}\ar[l]&J{\it D}^2(S^0_2M)\ar[u]^{J{\it D}^2(\phi_\lambda)}\ar[l]\ar[r]^{R}&S^0_2M\ar[u]_{S^0_2(\phi_\lambda)}}$}
\end{equation}

Taking into account that in coordinates $(x^i)_{1\le i\le n}$ on $M$, one has
\begin{equation}\label{ricci-flow-equation-a}
\left\{\begin{array}{l}
         g_{ij}(x,t)=(\partial x_i.\phi^a_t)(\partial x_j.\phi^b_t)\gamma_{ab}(\phi_t(x))\equiv
         (\phi_t)^a_i(x)(\phi_t)^b_j(x)\gamma_{ab}(\phi_t(x))\\
          (\partial t.g_{ij})(x,t)=(\dot\phi_t)^a_i(x)(\phi_t)^b_j(x)\gamma_{ab}(\phi_t(x))
          +(\phi_t)^a_i(x)(\dot\phi_t)^b_j(x)\gamma_{ab}(\phi_t(x))\\
          \hskip 2cm+(\phi_t)^a_i(x)(\phi_t)^b_j(x)(\dot\phi_t)^\alpha(x)(\partial x_\alpha.\gamma_{ab})(\phi_t(x))\\
          R_{ij}(x,t)=(\phi_t)^a_i(x)(\phi_t)^b_j(x)R_{ab}(\phi_t(x))\\
       \end{array}\right.
\end{equation}
we get that equation (\ref{ricci-flow-equation}) can be written as follows

\begin{equation}\label{ricci-flow-equation-b}
{}_\bullet(RF)\subset J{\it D}(W):\hskip 2pt\left\{\begin{array}{l}
          (\dot\phi_t)^a_i(x)(\phi_t)^b_j(x)\gamma_{ab}(\phi_t(x))
         +(\phi_t)^a_i(x)(\dot\phi_t)^b_j(x)\gamma_{ab}(\phi_t(x))\\
          \hskip 2cm+(\phi_t)^a_i(x)(\phi_t)^b_j(x)(\dot\phi_t)^\alpha(x)(\partial x_\alpha.\gamma_{ab})(\phi_t(x))\\
          =-2(\phi_t)^a_i(x)(\phi_t)^b_j(x)R_{ab}(\phi_t(x))\\
       \end{array}\right.
\end{equation}

where $\pi:W\equiv \mathbb{R}\times M^2\to\mathbb{R}\times M$, $(t,x^i,y^j)\mapsto(t,x^i)$.

Let us first introduce an equivalence relation in the set $\mathcal{S}^{\infty}_{ol}({}_\bullet(RF))$ of smooth solutions of ${}_\bullet(RF)$.
\begin{definition}\label{rigid-flow}
Let $(M,\gamma)$ be a Riemannian manifold of dimension $n$. We say that a diffeomorphism $f:M\to M$, is {\em rigid} if $f^*\gamma=\gamma$, and $\det(j(f))=1$, where $j(f)$ is the Jacobian matrix of $f$. In other words $(j(f)^k_i)\in SO(n)$. We say that a flow,  $\phi_t(x)$, solution of ${}_\bullet(RF)$, is {\em rigid}, if $\phi_t^*$, is a rigid diffeomorphism, $\forall t$. Then we say, also, that $\phi_t(x)$ is a {\em rigid solution} of ${}_\bullet(RF)$.
\end{definition}

\begin{definition}\label{equivalence-relation-up-rigid-flow}
We say that two flows $\phi_t(x)$ and $\psi_t(x)$, smooth solutions of ${}_\bullet(RF)$, are {\em rigid equivalent}, if there exists a rigid solution $f_t$ of ${}_\bullet(RF)$ such that $\phi_t=\psi_t\circ f_t$.
\end{definition}

\begin{proposition}\label{equivalence-relation-up-rigid-flow-prposition}
The rigid equivalence is an equivalence relation in $\mathcal{S}^{\infty}_{ol}({}_\bullet(RF))$, that we denote with $\thicksim_R$. Then for any representative $\psi_t\in[\phi_t]_R\in \mathcal{S}^{\infty}_{ol}({}_\bullet(RF))/\thicksim_R\equiv [\mathcal{S}^{\infty}_{ol}({}_\bullet(RF))]_R$, one has
$(\psi_t)^*\gamma=(\phi_t)^*\gamma$, and $\psi_t^*\eta=\phi_t^*\eta$, where $\eta$ is the canonical form associated to the metric $\gamma$.\footnote{The deformation in Fig.\ref{smale-paradox} cannot be rigid since it changes the orientation of the sphere $S^2$. In fact, this deformation can be realized by contracting the Est hemisphere, (resp. West hemisphere), endowed of a normal unitary vector $E$ (resp. $W$), on the middle of its equatorial-line, on the oriented disk $(D^2,E)$, (resp. $(D^2,W)$), with $\partial D^2$ passing for the poles $N$ and $S$. Then, continuing to deform the two oriented disks in the same opposite directions. During this deformation the boundary $\partial D^2$ remains fixed, and both hemispheres conserve their orientations, but the resulting sphere has reversed orientation.}
\end{proposition}

\begin{proof}
In fact, if $\phi_t(x)\in {}_\bullet(RF)$ it follows that  $\phi_t(x)\thicksim_R \phi_t(x)$, since $\phi_t=\phi_t\circ (id_M)_t$. If $\phi_t, \psi_t\in {}_\bullet(RF)$ and $\phi_t\thicksim_R\psi_t$ it follow that there exists a rigid solution $f_t\in {}_\bullet(RF)$ such that $\phi_t=\psi_t\circ f_t$. Since we can also write $\phi_t\circ f^{-1}=\psi_t$, and of course also $f^{-1}\in {}_\bullet(RF)$, it follows that $\psi_t\thicksim_R\phi_t$. Finally, if one has $\phi_t, \psi_t,\varphi_t\in {}_\bullet(RF)$ with $\phi_t\thicksim_R\psi_t\thicksim_R\varphi_t$, it follows that $\phi_t=\psi_t\circ (f_t)_1$, $\psi_t=\varphi_t\circ (f_t)_2$, for some rigid solutions $(f_t)_i$, $i=1,2$. Then we get also
$\phi_t=\varphi_t\circ (f_t)_2\circ (f_t)_1=\varphi_t\circ f_t$, where $f_t=(f_t)_2\circ (f_t)_1$ is also a rigid solution of ${}_\bullet(RF)$. Therefore, $\phi_t\thicksim_R\varphi_t$.
\end{proof}

Let us find solutions of equations(\ref{ricci-flow-equation-b}) of the type

$\phi_t^a=e^{\omega t}h^a(x)$. Then one has that $\omega$ and $h^a(x)$ must satisfy the following equations:

\begin{equation}\label{ricci-flow-equation-c}
          h^a_i(x)h^b_j(x)\left\{\omega\left[2\gamma_{ab}(\phi_t(x))+
          e^{\omega t}h^\alpha(x)(\partial x_\alpha.\gamma_{ab})(\phi_t(x))\right]+2R_{ab}(\phi_t(x))\right\}=0.
\end{equation}

In the case that $(M,\gamma)=(\mathbb{R}^n,\gamma_E)$, then $R_{ab}=0$ and above equations (\ref{ricci-flow-equation-c}) reduce to the following ones

\begin{equation}\label{ricci-flow-equation-d}
          h^a_i(x)h^b_j(x)\omega\left[2\gamma_{ab}(\phi_t(x))+
          e^{\omega t}h^\alpha(x)(\partial x_\alpha.\gamma_{ab})(\phi_t(x))\right]=0.
\end{equation}
Let us add the condition $\phi_0=id_M$, then we get that necessarily $h^a(x)=x^a$, for $a=1,\dots, n$.  Therefore, equations (\ref{ricci-flow-equation-d}) reduce to the following ones:

\begin{equation}\label{ricci-flow-equation-e}
          \omega\left[2\gamma_{ij}(\phi_t(x))+
          e^{\omega t}x^\alpha(\partial x_\alpha.\gamma_{ij})(\phi_t(x))\right]=0.
\end{equation}

Of course we can take a cartesian coordinate system, so that $\gamma_{ij}=\delta_{ij}$. Therefore, equations (\ref{ricci-flow-equation-e}) reduce to the following equations $2\omega\delta_{ij}=0$. Thus we get the unique trivial solution for $\omega$, i.e., $\omega=0$ and the unique solution for $\phi_t(x)$, i.e., $\phi_t^a=x^a$, that gives the unique solution $g(x,t)=\gamma_E(x)$.

Above result does not depend on the particular assumption made on the type of solutions. In fact, if consider equation (\ref{ricci-flow-equation-b}) in the case $M=\mathbb{R}^n$, with $\gamma=\delta_{ij}$, we get the following equations

\begin{equation}\label{ricci-flow-equation-f}
\left\{\begin{array}{l}
A^{rs}_{abij}(\partial t\partial x_r.\phi_t^a)(\partial x_s.\phi_t^b)=0\\
A^{rs}_{abij}\equiv\delta_{ab}(\delta^{rs}_{ij}+\delta^{rs}_{ji}).\\
\end{array}
\right.\end{equation}

Then, one can easily see that the unique smooth solution of (\ref{ricci-flow-equation-f}) that satisfies the initial condition $\phi^a_0=x^a$ is just $\phi_t^a=x^a$, for all $t\in\mathbb{R}$. (Of course this is up to rigid flows that do not produce any deformation.) In fact, let us write above equation for $1 \le i=j\le n$. We get the following equations:

\begin{equation}\label{ricci-flow-equation-g}
\left\{\begin{array}{l}
2\delta_{ab}f^a_{t1}f^b_1=0\\
\cdots\\
2\delta_{ab}f^a_{tn}f^b_n=0\\
\end{array}
\right.\end{equation}

Here, for simplicity, we have put $f\equiv\phi_t$. Then we can also write:
\begin{equation}\label{ricci-flow-equation-h}
\left\{\begin{array}{l}
\partial t.\left((\partial x_1.f^1)^2+\cdots+(\partial x_1.f^n)^2\right)=0\\
\cdots\cr
\partial t.\left((\partial x_n.f^1)^2+\cdots+(\partial x_n.f^n)^2\right)=0\\
\end{array}
\right.\end{equation}

Therefore, we get the following first integrals:
\begin{equation}\label{ricci-flow-equation-h}
\left\{\begin{array}{l}
(\partial x_1.f^1)^2+\cdots+(\partial x_1.f^n)^2=c_1(x^1,\cdots,x^n)\\
\cdots\cr
(\partial x_n.f^1)^2+\cdots+(\partial x_n.f^n)^2=c_n(x^1,\cdots,x^n)\\
\end{array}
\right.\end{equation}

This means that the jacobian matrix $(\partial x_i.(\phi_t)^j)=(j^j_i(x^1,\cdots,x^n))$ is made by functions that depend only on the coordinates $x^i$.
\begin{equation}\label{ricci-flow-equation-i}
(\partial x_i.(\phi_t)^j)=j^j_i(x^1,\cdots,x^n),\quad i,j=1,\cdots, n.
\end{equation}

Let us integrate (\ref{ricci-flow-equation-i}):

\begin{equation}\label{ricci-flow-equation-l}
(\phi_t)^j=\int j^j_i(x^1,\cdots,x^n)\, dx^i+c^j(t,x^1,\cdots,\hat x^i,\cdots,x^n),
\end{equation}
where $c^j(t,x^1,\cdots,\hat x^i,\cdots,x^n)$ is any function that does not depend on $x^i$.
Taking into account the initial condition $\phi^j_0=x^j$, we get
\begin{equation}\label{ricci-flow-equation-m}
x^j=\int j^j_i(x^1,\cdots,x^n)\, dx^i+c^j(0,x^1,\cdots,\hat x^i,\cdots,x^n).
\end{equation}
Let us again derive (\ref{ricci-flow-equation-m}) with respect to $x^i$, we get $\delta^j_i=j^j_i$. So we have
$(\phi_t)^j=x^j+c^j(t,x^1,\cdots,\hat x^i,\cdots,x^n)$. Since this holds for any $1\le i\le n$, we conclude that the arbitrary functions $c^j$ can depend only on $t$. So we have the following functions:
\begin{equation}\label{ricci-flow-equation-n}
(\phi_t)^j=x^j+c^j(t).
\end{equation}
The flows given in (\ref{ricci-flow-equation-n}) are rigid flows, therefore do not produce any deformation. Since, of course, we aim find solutions up to rigid ones, we get that in the case of the $n$-dimensional Euclidean manifold $(\mathbb{R}^n,\gamma)$, the unique complete noncompact solution of Ricci flow equation is just $g(t,x)=\gamma$.
\end{proof}

\section{\bf PINCHING PROBLEMS IN RICCI FLOW EQUATION}

 In the pinching problems connected to the Ricci flow equation one aims to characterize solutions of this equation on a connected compact Riemannian manifold $(M,\gamma)$ with respect to the so-called {\em pinching constant} $\lambda=\hbox{\rm min sec/max sec}$, i.e., by adding an upper as well as a lower bound on the sectional curvature (denoted {\rm sec}). With respect to this scalar-characterization of Riemannian manifolds, there are some interesting informations on their global structure. Let us recall some ones.\footnote{These results are usually referred as {\em sphere theorems}, or {\em quarter-pinched sphere theorems}. In fact, they are generalizations of the {\em sphere theorem for $3$-manifolds} that states that if $M$ is an orientable $3$-manifold with $\pi_2(M)\not=0$, then there exists a non-zero element of $\pi_2(M)$ having a representative that is an {\em embedding} $S^2\to M$.} (S. Brendle and R. Schoens \cite{BRENDLE-SCHOENS}) If $M$ is a compact $n$-dimensional Riemannian manifold with $\lambda\in({{1}\over{4}},1)$ ({\em strict pinching}), $M$ is diffeomorphic to a spherical space form. If $\lambda\in[{{1}\over{4}},1]$ $M$ is diffeomorphic to a spherical space form or isometric to a locally symmetric space.
 (M. Berger \cite{BERGER} and W. Klingenberg \cite{KLING}) If $M$ is a compact simply connected manifold with $\lambda\ge{{1}\over{4}}$ then $M$ is either homeomorphic to $S^n$ or isometric to $\mathbb{C}P^n$, $\mathbb{H}P^n$ or $Ca\mathbb{P}^2$, with their standard Fubini metric. [(J. L. Synge \cite{SYNGE}) A manifold with positive curvature does not necessitate be simply connected. In fact if $\dim M=n=2k$, one has $\pi_1(M)=0$ if orientable and $\pi_1(M)=\mathbb{Z}_2$ if non-orentable. If $\dim M=n=2k+1$ and positively curved $M$ is orientable.] (J. Cheeger \cite{CHEEGER}) Given a constant $\epsilon>0$, there are only finitely many diffeomorphism types of compact simply connected $2n$-dimensional manifolds $M$ with $\lambda\ge\epsilon$. (F. Fang and X. Rong \cite{FANG-RONG}, A. Petrunin and W. Tuschmann \cite{PETRUNIN-TUSCHMANN, TUSCH}) Given a constant $\epsilon>0$, there are only finitely many diffeomorphism types of compact $(2n+1)$-dimensional manifolds  $M$ with $\pi_1(M)=\pi_2(M)=0$ and $\lambda\ge\epsilon$.

Let us consider four-dimensional manifolds with {\em $\lambda$-pinched flag curvature}, ($0\le \lambda<1$), i.e., $R_u(v,v)\ge\lambda(x)R_u(w,w)$, where $R_u(.,.)=R(u,.,u,.)$ is the symmetric bilinear form, identified by the curvature tensor $R$ and any nonzero vector $u\in T_xM$. Here $v,w\in T_x^\perp<\mathbb{R}u>$, $|v|=|w|$. In \cite{ANDREWS-NGUYEN} B. Andrwes and H. Nguyen proved that the class of positively curved compact connected $4$-dimensional manifolds, with $\lambda$-pinched flag curvature, $\lambda\ge {{1}\over{4}}$, is invariant under Ricci-flow. Moreover, any such manifold is either diffeomorphic to a spherical space form or isometric to $\mathbb{C}P^2$ with Fubini-study metric (up to scaling).
In that interesting paper, authors refer to a uniqueness existence theorem for local solutions of the Cauchy problem for the Ricci-flow equation. (A first proof has been given by R. S. Hamilton \cite{HAMIL1, HAMIL2}. Different proofs were obtained also by D. T. De Turk \cite{DETURK} and B. Chow and D. Knopp \cite{CHOW-KNOPP}.) However, it is important to underline that such uniqueness is strictly related to the class of regular solutions considered there. In fact, by recasting the Cauchy problem in the geometric theory of PDE's, one can see that the Ricci-flow equation identifies an analytic submanifold $(RF)\subset J{\it D}^2(W)$ of the second jet-derivative space for sections of the following fiber bundle $\pi:W\equiv \mathbb{R}\times \widetilde{S^0_2(M)}\to \mathbb{R}\times M$, where $\widetilde{S^0_2(M)}\subset S^0_2(M)$ is the open subbundle of non-degenerate symmetric tensors of type $(0,2)$ on $M$. (See \cite{PRA14} to understand in which sense must be interpreted the statement about uniqueness of smooth solutions for the Ricci flow equation.) However, in order to describe singularities in the flow, it is useful to consider the embeddings $(RF)\subset J{\it D}^2(W)\subset J^2_{n+1}(W)$, where $J^2_{n+1}(W)$ is the $2$-jet-space for $(n+1)$-dimensional submanifolds of $W$, $\dim M=n$. Since the Ricci-flow equation is formally integrable and completely integrable, with non-trivial symbol $g_2$, there are also singular solutions satisfying smooth Cauchy problems.\footnote{It is well known that every compact smooth manifold, (or compactifiable $C^\infty$ manifold), can be considered analytic too. This is exentially the meaning of the famous results by J. Nash \cite{NASH} and some relative improvements by A. Tognoli \cite{TOGNOLI} and T. Kawakami \cite{KAWA}.} Then the characterization of global solutions is obtained by means of the singular integral bordism groups of $(RF)$. (For details on the geometry of PDE's see \cite{PRA1, PRA2, PRA3, PRA4, PRA5, PRA6, PRA7, PRA8, PRA9, PRA10, PRA11, PRA14}.) For example, in the particular case of $4$-dimensional closed, compact, simply connected, smooth Riemannian manifolds $M$, one obtains $\Omega_4^{(RF)}/K^{(RF)}_4\cong\mathbb{Z}_2\oplus\mathbb{Z}_2\cong \Omega_4$, where $K^{(RF)}_4$ is the kernel of the canonical projection $p:\Omega_4^{(RF)}\to\Omega_4\cong \mathbb{Z}_2\bigoplus\mathbb{Z}_2$. With this respect, and taking into account the well-known theorems by A. Dold \cite{DOLD} and C. T. C. Wall \cite{WALL} on the (oriented) cobordism ring (${}^{+}\Omega_\bullet$) $\Omega_\bullet$, we can get from the results of this paper the following interesting theorem.

\begin{theorem}\label{}
The $4$-dimensional Riemannian manifolds preserving $\lambda$-pinched flag curvature, $\lambda\ge {{1}\over{4}}$, in a Ricci-flow, belong to the cobordism classes in the image $p^{-1}(r({}^{+}\Omega_4\cong\mathbb{Z}))\subset \Omega_4^{(RF)}$, where $r:{}^{+}\Omega_\bullet\to\Omega_\bullet$ is the forgetting orientation natural mapping, i.e. one has the exact sequences {\em(\ref{exact-commutative-diagram-oriented-non-oriented-bordisms-integral-bordism})}.
\begin{equation}\label{exact-commutative-diagram-oriented-non-oriented-bordisms-integral-bordism}
 \xymatrix{&&&\fbox{${}^+\Omega_4\cong\mathbb{Z}$}\ar[d]^{2}&\\
 &&&\fbox{${}^+\Omega_4\cong\mathbb{Z}$}\ar[d]^{r}&\\
 0\ar[r]&K_4^{(RF)}\ar[r]&\Omega_4^{(RF)}\ar[r]_(0.3){p}&\fbox{$\Omega_4\cong\mathbb{Z}_2\bigoplus\mathbb{Z}_2$}\ar[r]&0}.
\end{equation}
\end{theorem}

\section{\bf NECKPINCHING PROBLEMS IN SINGULAR EXOTIC HEAT PDE}\label{neckpinching-problems-in-singular-exotic-heat-pde}

In a recent paper Z. Gang and I. M. Sigal \cite{GANG-SIGAL} considered some particular solutions for the Ricci flow equation encoding the mean curvature flow of an initial hypersurface $M_0\subset \mathbb{R}^{d+1}$, that is of revolution around the axis $x=x_{d+1}$. There $u(x,t)$ represents, at fixed $t$, the ''distance function'' of the revolution hypersurface from the revolution axis at $x\in\mathbb{R}$. More precisely they considered the following boundary value problem:
\begin{equation}\label{singular-exotic-heat-pde}
\fbox{$\begin{array}{l}
  \hbox{\rm(mean curvature flow equation):}\hskip 2pt u_t={{u_{xx}^2}\over{1+(u_x)^2}}-{{d-1}\over{u}} \\
  \\
\hbox{\rm(boundary conditions):}\, \left\{\begin{array}{l}
u(x,0)=u_0(x)>0,\hskip 2pt \forall x\in\mathbb{R}\\
\lim \inf_{|x|\to\infty} u_0(x)>0\\
||{{1}\over{u(x,t)}}||_\infty<\infty, t<t^*\\
||{{1}\over{u(x,t)}}||_\infty\to\infty, t\to t^*\\
                         \end{array}\right.\\
\end{array}$}
\end{equation}

They proved, by using classic method of functional analysis, existence of solutions that in finite time collapse on the revolution axis. They called such solutions {\it collapsing} (or {\em neckpinching}) at the time $t^*$.

It is interesting to emphasize that translating above problem in the framework of the geometry of PDE's, the mean curvature flow equation in (\ref{singular-exotic-heat-pde}), identifies a real analytic $7$-dimensional submanifold $E_2$ of the $2$-jet-derivative space $J{\it D}^2(W)\cong \mathbb{R}^8$, $(t,x,u,u_t,u_x,u_{tt}, u_{tx},u_{xx})$, over the trivial vector bundle $\pi:W\equiv \mathbb{R}^3\to \mathbb{R}^2$, $(t,x,u)\mapsto (t,x)$. One has $E_2\cap S_2=\emptyset$, where $S_2\subset J{\it D}^2(W)$ is the analytic submanifold identified by the constraint $u=0$. So, a solution $V_\bullet\subset E_2$, such that $u(x,t^*)=0$, for some time $t^*$, cannot exist. However, this does not exclude that by considering the {\it singular equation} ${}^\bullet E_2\equiv E_2\bigcup S_2\subset J{\it D}^2(W) $, we can find {\it asymptotic solutions} of the above problem. This is in fact the case when $S_2$ is endowed with the distribution obtained from the Cartan distribution of $J{\it D}^2(W)$, just restricted on $S_2$. Then one can see that one can find solutions of $E_2$ that approach integral manifolds of $S_2$, when $u\to 0$. The tangent space to such solutions have zero time component in the soldering points. This just means that such solutions stop (or collapse), in a finite time $t^*$, to the revolution axis. (For informations on this geometric approach to singular PDE's see \cite{PRA3, PRA14, PRA15, AG-PRA2}.)

The Cartan distribution $\mathbf{E}_2\subset TE_2$ of $E_2$, is given by the following vector fields
\begin{equation}\label{cartan-distribution-2-a}
\begin{array}{ll}
  \zeta&=X^t(\partial t+u_t\partial u+u_{tt}\partial u_t+u_{tx}\partial u^x)\\
  &+X^x(\partial x+u_x\partial u+u_{xt}\partial u_t+u_{xx}\partial u^x)\\
  &+Z_{xx}\partial u^{xx}+Z_{tt}\partial u^{tt}+Z_{tx}\partial u^{tx}\\
  \end{array}
\end{equation}

such that
\begin{equation}\label{cartan-distribution-2-b}
\left\{
\begin{array}{l}
X^t\left[u_t(u_t(1+u_x^2)-u^2_{xx})+u_{tt}(u(1+u^2_x))+u_{tx}(2uu_tu_x+2u_x(d-1))\right] \\
+ X^x\left[u_x(u_t(1+u_x^2)-u^2_{xx})+uu_{tx}(1+u^2_x)+u_{xx}(2uu_tu_x+2u_x(d-1))\right] \\
+2Z_{xx}uu_{xx}=0\\
\end{array}
\right\}
\end{equation}
with
\begin{equation}\label{cartan-distribution-2-ab}
 u_t=\frac{uu_{xx}^2-(1+u^2_x)(d-1)}{u(1+u^2_x)}.
 \end{equation}
For $u\to 0$ equation (\ref{cartan-distribution-2-b}) becomes $0=X^t[(1+u^2_x)(d-1)]^2$, hence we get $X^t=0$. Therefore, for $u\to 0$ we can write (\ref{cartan-distribution-2-a}) in the form given in (\ref{cartan-distribution-2-aa}).
\begin{equation}\label{cartan-distribution-2-aa}
  \zeta=X^x(\partial x+u_x\partial u+u_{xt}\partial u_t+u_{xx}\partial u^x)+Z_{xx}\partial u^{xx}+Z_{tt}\partial u^{tt}+Z_{tx}\partial u^{tx}.
\end{equation}
The Cartan distribution on $S_2$, (i.e., vector fields (\ref{cartan-distribution-2-a}) satisfying the condition $\zeta.u=0$), is given in (\ref{cartan-distribution-2-abb}).
\begin{equation}\label{cartan-distribution-2-abb}
\left\{\begin{array}{ll}
  \zeta&=X^x\left[-\frac{u_x}{u_t}\partial t+\partial x+(u_{tx}-\frac{u_x}{u_t}u_{tt})\partial u_t+(u_{xx}-\frac{u_x}{u_t}u_{tx})\partial u^x\right]\\
  &+Z_{xx}\partial u^{xx}+Z_{tt}\partial u^{tt}+Z_{tx}\partial u^{tx}.\\
  \end{array}\right.
\end{equation}
Therefore, when $q\in S_2$ approaches equation $E_2$, i.e., when $u_t$ approaches the function given in (\ref{cartan-distribution-2-ab}), one has $\lim_{q\to E_2}\mathbf{E}_2(S_2)_q=<\zeta>$, where $\zeta$ are given by the vector fields in (\ref{cartan-distribution-2-abbb})
\begin{equation}\label{cartan-distribution-2-abbb}
  \zeta=X^x\left[\partial x+u_{tx}\partial u_t+u_{xx}\partial u_x\right]+Z_{xx}\partial u^{xx}+Z_{tt}\partial u^{tt}+Z_{tx}\partial u^{tx}.
\end{equation}
This means that $\lim_{q\to E_2}\mathbf{E}_2(S_2)_q\subset\mathbf{E}_2$. This proves that solutions of $E_2$ can be prolonged to an integral manifold of $S_2$. Since, in the asymptotic limit, both distributions have not time components, then the solutions of $E_2$ approach $S_2$ to a fixed $t^*$ and stop on the revolution axis.

By resuming, we get the following theorem, proved by using geometric methods only.

\begin{theorem}\label{schroedinger-equation-global-attractor}
The singular boundary value problem {\em(\ref{singular-exotic-heat-pde})} admits smooth solutions, i.e., smooth solutions of the regular equation $E_2\subset J{\it D}^2(W)$ that collapse in a finite time approaching the set $S_2\subset J{\it D}^2(W) $ of singular points.
\end{theorem}

\section{\bf EXOTIC HEAT-SCR\"ODINGER EQUATION}

In \cite{MOLINET} it is studied the following so-called {\em cubic nonlinear Schr\"odinger equation} (NLS):
\begin{equation}\label{heat-schroedinger-equation}
\left\{\begin{array}{l}
 u_t+\gamma u+i(u_{xx}\mp |u|^2u)=f\\
u(t,x)\in\mathbb{C},\quad\forall(t,x)\in\Omega\equiv\mathbb{R}_+\times \mathbb{T},\hskip 2pt \mathbb{T}\equiv \mathbb{R}/2\pi\mathbb{Z}\\
\end{array}\right.
\end{equation}

where $\gamma>0$ ({\em damping parameter}) and $f\in L^2(\mathbb{T})$ ({\em time-independent-forcing term}). There the main result is that the nonlinear group $S(\cdot)$, associated to (\ref{heat-schroedinger-equation}), i.e., $S(t)(u_0)=u(t)$, $t\in\mathbb{R}$, where $u$ is the solution of (\ref{heat-schroedinger-equation}), for to the initial condition $u_0$, provides an infinite-dimensional dynamic system in $L^2(\mathbb{T})$ that has a global attractor $\mathcal{A}\subset H^2(\mathbb{T})$, that is a connected and compact set of $H^2(\mathbb{T})$, invariant (positively and negatively) by $S(\cdot)$ that attracts for the $L^2(\mathbb{T})$-metric all positive orbits uniformly with respect to bounded sets of initial data in $L^2(\mathbb{T})$.

In the framework of the algebraic topology of PDE's,  the interesting problem considered in that paper, has an intriguing issue. With this respect, let us recast equation (\ref{heat-schroedinger-equation}) into a real analytic PDE of second order over the following trivial vector fiber bundle $\pi:W\equiv M\times \mathbb{R}^2\to M\equiv \mathbb{R}_+\times\mathbb{T}$, $(t,x,v,w)\mapsto(t,x)$:
\begin{equation}\label{real-heat-schroedinger-equation}
  \fbox{$E_2\subset J\mathcal{D}^2(W)\hskip 2pt:\left\{
\begin{array}{l}
 v_t+\gamma v-w_{xx}\mp w(v^2+w^2)=\phi\\
 w_t+\gamma w+v_{xx}\mp v(v^2+w^2)=\psi\\
\end{array}
\right.$}
\end{equation}

where $u=v+iw$ and $f=\phi+i\psi$. Here we assume that $\phi$ and $\psi$ are analytic functions. Then one can prove that $E_2$ is an involutive, formally integrable and completely integrable PDE. In fact, one has
\begin{equation}\label{formal-integrability-real-heat-schroedinger-equation}
\fbox{$\dim(E_2)_{+1}=16$}=\fbox{$\dim E_2=12$}+\fbox{$\dim(g_2)_{+1}=4$}.
\end{equation}
The relation (\ref{formal-integrability-real-heat-schroedinger-equation}) proves that the canonical mapping $(E_2)_{+1}\to E_2$ is surjective. Furthermore, since $\dim(g_2)_{+1}=4=\dim g_2$, we get that $g_2$ is an involutive symbol. This facts are enough to state that $E_2$ is formally integrable, and since it is analytic it is completely integrable too. This means that in the neighborhood of any point $q\in E_2$, ({\em initial condition}), passes a regular solution (analytic solution). Furthermore, since $E_2$ satisfies some conditions of regularity, we can solve Cauchy problems for $1$-dimensional integral manifolds $N\subset E_2$, diffeomorphically projected in $W$, by means of the canonical projection $\pi_{2,0}:E_2\to W$. We call {\em admissible} such Cauchy manifolds. Note that solutions passing through admissible Cauchy manifolds do not necessitate to be smooth, in fact, in general, are singular ones with respect to the embeddings $E_2\subset J{\it D}^2(W)\subset J_2^2(W)$, where $J_2^2(W)$ denotes the $2$-jet-space for $2$-dimensional  submanifolds of $W$.

Furthermore, one can see that the $1$-dimensional integral singular bordism group of $E_2\subset J^2_2(W)$ is $\Omega_{1,s}^{E_2}\cong\mathbb{Z}_2$. Then for any two space-like smooth (or analytic) $1$-dimensional closed admissible Cauchy hypersurfaces $N_1\subset (E_2)_{t_1}$ and $N_2\subset (E_2)_{t_2}$, $t_1\not= t_2$, where $(E_2)_t\equiv \bar\pi_2^{-1}(t)$, with $\bar\pi$ the natural projection $E_2\to \mathbb{R}_+$, exists a singular solution $V\subset E_2$, such that $\partial V=N_1\sqcup N_2$, iff $[N_1\sqcup N_2]=[0]\in\Omega_{1,s}^{E_2}$. In order that this condition should be satisfied, it is enough that $N_1$ and $N_2$ have the same integral characteristic numbers of second order. (The solution $V$ bording $N_1$ and $N_2$ is a smooth solution iff above condition holds for all the orders, i.e. for all the conservation laws of $E_2$.) So, in general, there is not solution unicity for any admissible, $1$-dimensional closed smooth space-like Cauchy manifold, $N\subset E_2$, but all such solutions bord ones belonging to the same $0$-class in $\Omega_{1,s}^{E_2}\cong\mathbb{Z}_2$. This agrees with the main result in \cite{MOLINET}. By resuming we get the following theorem.\footnote{Here we have explicitly considered only the analytic case. However, it is also possible to extend such result to singular PDE's, similarly to what made in the previous Section \ref{neckpinching-problems-in-singular-exotic-heat-pde}.}

\begin{theorem}
Equation {\em(\ref{real-heat-schroedinger-equation})}, considered in the analytic case, has a {\em global attractor}, in the sense that all its (singular) solutions bord with a same integral bordism class.
\end{theorem}

\section{\bf EXOTIC SINGULAR VECTOR HEAT EQUATION}

In \cite{DASKA-SESUM} it is studied the PDE reported in (\ref{exotic-vector-heat-equation}).

\begin{equation}\label{exotic-vector-heat-equation}
(\partial t.F^k)+\kappa\, \nu^k=0
\end{equation}

where $F^k=F^k(t,u^a)\equiv F^k_t$, is a family of parametric equations for embeddings $F_t:\Sigma^2\to\mathbb{R}^3$ of $2$-dimensional, convex, star-shaped, compact surface $\Sigma^2$ into $\mathbb{R}^3$, with $(u^a)_{1\le a\le 2}$ local coordinates on $\Sigma^2$ and $(x^k)_{1\le k\le 3}$, coordinates in $\mathbb{R}^3$. $\kappa=\kappa(t,u^a)\equiv \hbox{\rm Gauss curvature}/\hbox{\rm mean curvature}$ denotes the so-called {\em harmonic mean curvature} of $\Sigma_t\equiv F_t(\Sigma^2)\subset\mathbb{R}^3$. $\nu=\nu^k\partial x_k$ is the unitary normal vector field on $\Sigma_t$. In \cite{DASKA-SESUM} it is assumed the mean curvature positive. The main result there concerns an existence of solutions $\Sigma_t^\epsilon$ for short time starting from a $\Sigma^2$ of class $C^{2,1}$ and a maximal time $T_\epsilon$ of existence of a smooth solution such that the mean curvature goes to zero as $t\to T_\epsilon$ at some point of $\Sigma_t^\epsilon$, or $\Sigma_t^\epsilon$ shrinks to a point as $t\to T_\epsilon$. In the particular case where $\Sigma^2$ is a surface of revolution, the flow always exists up to the time when the surface shrinks to a point.
It is interesting to recast this problem in the framework for geometric theory of PDE's, i.e., to implement this problem on a trivial vector fiber-bundle $\pi:W\equiv \mathbb{R}^6\to\mathbb{R}^3$, $(t,u^a,x^k)\mapsto(t,u^a)$. Then the problem is encoded by the PDE reported in (\ref{exotic-vector-heat-equation-b}).
\begin{equation}\label{exotic-vector-heat-equation-b}
\fbox{$  E_3\subset J{\it D}^3(W):\hskip 2pt\left\{
\begin{array}{l}
 A(x^i,x^i_a,x^i_{ab})x^k_t+B^k(x^i,x^i_a,x^i_{ab})=0\hskip 3pt\hbox{\rm(HMCF)}\\
b_{ab;c}-b_{ac;b}=0\hskip 3pt\hbox{\rm(Gauss-Codazzi equation)}\\
b_{ab}=[\epsilon_{ijk}x^i_{ab}x^j_{\bar a}x^k_{\bar b}\delta^{\bar a}_1\delta^{\bar b}_2]/\sqrt{EG-F^2}.\\
\end{array}
\right.$}
\end{equation}

In (\ref{exotic-vector-heat-equation-b}) $A$ and $B^k$ are known analytic functions of their arguments, $b_{ab}$ are the components of the second fundamental form and $E,F,G$, are the usual Gauss symbols of the surface $\Sigma_t$. The semi-colon, in the Gauss-Codazzi equation denotes tensor derivative with respect to $u^c$.  In general the PDE in (\ref{exotic-vector-heat-equation-b}) is a singular equation, however in the case considered in \cite{DASKA-SESUM}, i.e., positive mean curvature, one has $A\not=0$. Then, in order to characterize local solutions, it is important to study the formal properties of such equation, and, by means the determination of its integral bordism groups, characterize global solutions also. In this way one could generalize the following well known result of differential geometry: {\em Given symmetric functions $\gamma_{ab}=\gamma_{ab}(u^c)$ and $b_{ab}=b_{ab}(u^c)$, $a,b,c=1,2$, such that the Gauss-Codazzi equations are satisfied, there exists a surface $x^i=x^i(u^c)$, uniquely determined up to rigid motions, such that $\gamma$ and $b$ are respectively the first fundamental form and the second fundamental form of such a surface.} (See, e.g. L. P. Eisenhart \cite{EISEN}.) Of course this agrees with the structure of equation in (\ref{exotic-vector-heat-equation-b}), since there coordinates with time derivatives are determined by the other ones containing coordinates with space derivatives only. It is important to emphasize that equation (\ref{exotic-vector-heat-equation-b}) is not formally integrable, since there is not surjectivity between the first prolongation $(E_3)_{+1}$ and $E_3$. In fact,
\begin{equation}\label{non-formal-integrability-vector-singular-pde}
\fbox{$\dim((E_3)_{+1})=88$}<\fbox{$\dim (E_3)=58$}+\fbox{$\dim((g_3)_{+1})=41$}.
\end{equation}

However, by using the geometric theory of PDE's, it is possible to identify a subequation $\widehat{(E_3)}\subset E_3$ that is formally integrable and completely integrable. For this it is enough to add to the equations in (\ref{exotic-vector-heat-equation-b}) also the first prolongations of (HMCF). More precisely, we shall use equation (\ref{exotic-vector-heat-equation-c}).

\begin{equation}\label{exotic-vector-heat-equation-c}
\fbox{$  \widehat{E_3}\subset J{\it D}^3(W):\hskip 2pt\left\{
\begin{array}{l}
 Ax^k_t+B^k=0\\
 A x^k_{tt}+[(\partial x^a_i.A)x^i_{at}+(\partial x^{ab}_i.A)x^i_{abt}]x^k_t\\
 \hskip 5pt+(\partial x^a_i.B^k)x^i_{at}+(\partial x^{ab}_i.B^k)x^i_{abt}=0\\
  A x^k_{tc}+[(\partial x^a_i.A)x^i_{ac}+(\partial x^{ab}_i.A)x^i_{abc}]x^k_t\\
   \hskip 5pt+(\partial x^a_i.B^k)x^i_{ac}+(\partial x^{ab}_i.B^k)x^i_{abc}=0\\
b_{ab;c}-b_{ac;b}=0\\
{[b_{ab;c}-b_{ac;b}]_{,t}=0}\\
{[b_{ab;c}-b_{ac;b}]_{,d}=0}\\
{b_{ab}=[\epsilon_{ijk}x^i_{ab}x^j_{\bar a}x^k_{\bar b}\delta^{\bar a}_1\delta^{\bar b}_2]/\sqrt{EG-F^2}}.\\
\end{array}
\right.$}
\end{equation}
Then $\widehat{E_3}_{+1}\to\widehat{E_3}$ is surjective as proved in (\ref{formal-integrability-vector-singular-pde}).
\begin{equation}\label{formal-integrability-vector-singular-pde}
\fbox{$\dim((\widehat{E_3})_{+1})=70$}=\fbox{$\dim \widehat{E_3}=49$}+\fbox{$\dim((\widehat{g_3})_{+1})=21$}.
\end{equation}

Furthermore, the symbol $\widehat{g_3}$ is involutive, since one has $\dim(\widehat{g_3}_{+1})=21=\dim(\widehat{g_3})$. Therefore $\widehat{E_3}$ is formally integrable and completely integrable.

Then the characterization of Cauchy problems and global solutions, weak, singular and smooth, can be directly obtained by using the algebraic topologic methods in PDE's  introduced by A. Pr\'astaro. The asymptotic behaviour of solutions as mean curvature approaches the zero value, can be obtained by considering the singular PDE in (\ref{singular-vector-singular-pde})
\begin{equation}\label{singular-vector-singular-pde}
\fbox{$\widehat{(E_3)}^{(S)}\equiv \widehat{(E_3)}\bigcup S_3\subset J{\it D}^3(W):\hskip 2pt\left\{\begin{array}{l}
                                                                      \widehat{(E_3)} \hskip 0.5cm\hbox{\rm(\ref{exotic-vector-heat-equation-c})}  \\                 S_3\equiv\{{{1}\over{2}}\gamma^{ab}b_{ab}=0\}\\
                                                                      \end{array}\right.$}
\end{equation}

i.e., adding to equation (\ref{exotic-vector-heat-equation-c}) the submanifold $S_3$ of $J{\it D}^3(W)$, identified by the condition of zero mean-curvature. (See, e.g. Section \ref{neckpinching-problems-in-singular-exotic-heat-pde} where similar problems are considered.)

\section{\bf COMPLEX RICCI FLOW EQUATION AS A QUANTUM EQUATION AND POINCAR\'E CONJECTURE}

In (\ref{complex-ricci-flow-equation}) it is reported the so-called {\em normalized K\"ahler-Ricci flow equation}
\begin{equation}\label{complex-ricci-flow-equation}
\scalebox{0.9}{$ \left\{\begin{array}{l}
\hbox{\rm(normalized K\"ahler-Ricci flow equation)}:\hskip 2pt \partial_t.g(t)=-Ric(g(t))+g(t)\\
\hbox{\rm $g(t)$ is a K\"ahler metric for each $t$;}\\
\hbox{\rm $\omega(t)=\omega_{g(t)}$, the K\"ahler form associated to $g(t)$, satisfies the condition}\\
\hskip 5pt[\omega(t)={{i}\over{2\pi}}g(t)_{\alpha\bar\beta}dz^\alpha\wedge d\bar z\beta ]=c_1(M)
=[{{i}\over{2\pi}}R_{\alpha\bar\beta}(g(t))dz^\alpha\wedge d\bar z\beta =\rho(t)]>0.\\
\end{array}\right.$}
\end{equation}
Here $M$ is a compact K\"ahler manifold, with the first Chern  class $c_1(M)>0$. (Positivity of the first Chern class is in the sense of K. Kodaira \cite{KODAIRA}.) $\omega(t)$ denotes the K\"ahler form associated to $g(t)$ and $\rho(t)$ is the Ricci form that represents also the first Chern class of $M$. Some authors study solutions of (\ref{complex-ricci-flow-equation}) by adding also the global constraint in (\ref{bounded-problem}).
\begin{equation}\label{bounded-problem}
   \int_M|Rm(g(t))|^{n/2}\, dv_t\le C,\hskip 2pt \dim_{\mathbb{C}}M=n\ge 3,\hskip 2pt C\in\mathbb{R}
\end{equation}

There $Rm(g(t))$ denotes the curvature operator. (See, e.g., the work by P. Daskalopoulos and M. Sesum \cite{DASKA-SESUM}, and papers quoted there.) The main result in \cite{DASKA-SESUM} is an improvement of a previous result by Sesum, where was also made the hypothesis of bounded Ricci curvature: $|R(g(t)|\le C$. More precisely if $g(t)$, $t\in[0,\infty)$ is a solution of the problem (\ref{bounded-problem}) that at $t=0$ has curvature operator uniformly bounded in $L^n$-norm, the curvature operator will also be uniformly bounded along the flow. This is enough to state that the flow will converge along a subsequence to a K\"ahler-Ricci soliton, i.e., a compact K\"ahler manifold $(N,h)$, such that its K\"ahler form $\omega_h$ satisfies the following equation: $Ric(\omega_h)-\omega_h=\mathcal{L}_\zeta\omega_h$, where $\zeta$ is a holomorphic vector field on $N$, and $\mathcal{L}_\zeta$ denotes Lie derivative with respect to $\zeta$. It follows that the first Chern class $c_1(N)$ of $N$ is positive and represented by $\omega_h$. Since there are no K\"ahler-Einstein metrics on $N$ if $N$ admits a K\"ahler-Ricci soliton, (this is a result by A. Futaki \cite{FUTAKI}), it follows that the existence of K\"ahler-Ricci soliton is an obstruction to the existence of K\"ahler-Einstein metrics on compact K\"ahler manifolds with positive first Chern class. The condition $c_1(M)>0$ is important, after the result by S.-T. Yau \cite{YAU1}, on the existence of K\"ahler-Einstein metrics on K\"ahler manifolds with $c_1(M)\le 0$, and by T. Aubin \cite{AUBIN} for K\"ahler manifolds with $c_1(M)<0$. (This was first conjectured by E. Calabi \cite{CALABI}. For complementary informations see, e.g., \cite{YAU2}.)

It is interesting to emphasize that the problem considered in (\cite{DASKA-SESUM}) can be recast in the geometric theory of PDE, directly working in the category of complex manifolds. This can be made, by using the geometric theory of quantum PDE's formulated by A. Pr\'astaro. In fact the category of complex manifolds can be considered a little subcategory of the one for quantum manifolds, (in the sense introduced by A. Pr\'astaro), where the quantum algebra coincides with the algebra of complex numbers and quantum differentiable functions are identified with holomorphic functions. Then the system (\ref{complex-ricci-flow-equation}) can be encoded with the following second order PDE $\widetilde{E}_2$ in the category of complex manifolds (like quantum manifolds):
\begin{equation}\label{complex-ricci-flow-equation-b}
\scalebox{0.8}{$\fbox{$\widetilde{E}_2\subset J\widetilde{\mathcal{D}}^2(\widetilde{W})\subset \widetilde{J}_{n+1}^2(\widetilde{W}):\left\{
\begin{array}{ll}
 \hbox{\rm(1(a))}&:\hskip 2pt (\partial_t.g)_{\alpha\bar\beta}(t)=-Ric(g(t))_{\alpha\bar\beta}+g(t)_{\alpha\bar\beta}\\
\hbox{\rm(1(b))}&:\hskip 2pt R_{\alpha\bar\beta}=-(\partial z_\alpha\partial\bar z_\beta.\ln\det(g_{\delta\bar\epsilon}))\\
\hbox{\rm (2(a))}&:\hskip 2pt  g_{\alpha\bar\beta}=g_{\bar\beta\alpha}\\
\hbox{\rm (2(b))}&:\hskip 2pt  \bar g_{\alpha\bar\beta}=g_{\bar\alpha\beta}\\
\hbox{\rm (3(a))}&:\hskip 2pt (\partial z_\gamma.g_{\alpha\bar\beta})-(\partial z_\alpha.g_{\gamma\bar\beta})=0\\
\hbox{\rm (3(b))}&:\hskip 2pt  (\partial \bar z_\gamma.g_{\alpha\bar\beta})-(\partial \bar z_\beta.g_{\alpha\bar\gamma})=0\\
\hbox{\rm (4)}&:\hskip 2pt g_{\alpha\bar\beta}-R_{\alpha\bar\beta}-(\partial z_\alpha\partial\bar z_\beta.f)=0\\
\end{array}\right.$}$}
\end{equation}
In (\ref{complex-ricci-flow-equation-b}) $\widetilde{W}$ is the following fiber bundle $\widetilde{\pi}:\widetilde{W}\equiv\mathbb{R}\times\widetilde{S^0_2M}\times\mathbb{C}\to \mathbb{R}\times M$, $(t,z^\alpha,\bar z^\alpha, u_{\alpha\bar\beta},\lambda)\mapsto(t,z^\alpha, \bar z^\alpha)$. (Local) sections of $\widetilde{\pi}$ are identified with the following (local) functions $\{g_{\alpha\bar\beta}(t,z^\alpha,\bar z^\alpha),f(t,z^\alpha,\bar z^\alpha)\}$. The ''tilde'' over the jet(-derivative) spaces denotes ''holomorphic class of differentiability''. Our geometric theory of quantum PDE's allows us to obtain solutions as $(n+1)$-dimensional complex integral submanifolds of $\widetilde{E}_2$. Then the characterization of global solutions is made by means of the integral bordism group $\Omega_{n}^{\widetilde{E}_2}$. With this respect one can generalize the Poincar\'e conjecture in this category of complex manifolds. (For details see \cite{PRA12} where a generalized version of the Poincar\'e conjecture is formulated, and proved too, in the category of quantum (super)manifolds.)

\end{document}